\documentclass[a4paper]{amsart}
\usepackage[utf8]{inputenc}
\usepackage{amsmath}
\usepackage{amsfonts}
\usepackage{amssymb}

\usepackage{thmtools}

\usepackage{amsthm}
\usepackage{cancel}
\usepackage{graphicx}
\usepackage{multicol}

\theoremstyle{plain}
\declaretheorem[name={Theorem},numberwithin=section]{thm}
\declaretheorem[name={Proposition},sibling=thm]{proposition}
\declaretheorem[name={Claim}, numberwithin=thm]{claim}
\declaretheorem[name={Corollary},sibling=thm]{corollary}
\declaretheorem[name={Lemma},sibling=thm]{lemma}

\newcommand{\dA}{\mathbf{\mathcal{A}}}
\newcommand{\dB}{\mathbf{\mathcal{B}}}

\newcommand{\R}{\mathcal{R}}
\newcommand{\A}{\mathcal{A}}

\newcommand{\C}{\mathcal{C}}

\theoremstyle{definition}
\newtheorem{definition}{Definition}[section]
\newtheorem*{remark}{Convention}

\theoremstyle{plain}

\theoremstyle{plain}
\theoremstyle{definition}

\theoremstyle{plain}
\newtheorem*{lemma*}{Lemma}

\newcommand{\set}[1]{\{{#1}\}}

\DeclareMathOperator{\range}{range}

\DeclareMathOperator{\Spec}{DgSp}

\DeclareMathOperator{\Id}{Id}

\title[Reductions between equivalence relations]{Measuring the complexity of reductions\\ between equivalence relations}

\author[Fokina]{Ekaterina Fokina}
\address{Institute of Discrete Mathematics and Geometry, Vienna University of Technology, Austria}
\email{ekaterina.fokina@tuwien.ac.at}
\urladdr{dmg.tuwien.ac.at/fokina}

\author[Rossegger]{Dino Rossegger}
\address{Institute of Discrete Mathematics and Geometry, Vienna University of Technology, Austria}
\email{dino.rossegger@tuwien.ac.at}
\urladdr{dmg.tuwien.ac.at/rossegger}

\author[San Mauro]{Luca San Mauro}
\address{Institute of Discrete Mathematics and Geometry, Vienna University of Technology, Austria}
\email{luca.san.mauro@tuwien.ac.at}
\urladdr{dmg.tuwien.ac.at/sanmauro}

\thanks{The authors were supported by the Austrian Science Fund FWF, project~P~27527. San Mauro was also partially supported by the Austrian Science Fund FWF, project~M~2461.}


\begin{document}

\maketitle

\begin{abstract}
Computable reducibility is a well-established notion that allows to compare the complexity of various equivalence relations over the natural numbers. We generalize computable reducibility by introducing degree spectra of reducibility and bi-reducibility. These spectra provide a natural way of measuring the complexity of reductions between equivalence relations.
We prove that any upward closed collection of Turing degrees with a countable basis can be realised as a reducibility spectrum or as a bi-reducibility spectrum. We show also that there is a reducibility spectrum of computably enumerable equivalence relations with no countable basis and a reducibility spectrum of computably enumerable equivalence relations which is downward dense, thus has no basis.
\end{abstract}

\section{Introduction}
 Computable reducibility is a long-standing notion that has proven to be  fruitful
for ranking the complexity
of equivalence relations over the set $\omega$ of natural numbers. The following definition is the relativised version of the one commonly considered in the literature.

\begin{definition}
Let $R$ and $S$ be two equivalence relations on $\omega$, and let $\mathbf{d}$ be a Turing degree. $R$ is \emph{$\mathbf{d}$-computably reducible} to $S$ (notation: $R \leq_{\mathbf{d}} S$), if there is a $\mathbf{d}$-computable function $f$ such that, for all natural numbers $x,y$, the following holds
\[
x R y \Leftrightarrow f(x)Sf(y).
\]
If $R \leq_{\mathbf{d}} S$ and $S \leq_{\mathbf{d}} R$, we write $R\equiv_{\mathbf{d}} S$, and we say that $R$ and $S$ are \emph{bi-reducible by $\mathbf{d}$}.
\end{definition}
The case $\mathbf{d}=\mathbf{0}$ has been thoroughly explored. The standard underlying intuition is that, if $R\leq_{\mathbf{0}} S$ via some $f$, then $S$ is at least as complex as $R$,
%(modulo the information encoded by $f$),
since all that is needed  to decide whether $x$ and $y$ are $R$-equivalent is to know if $f(x)$ and $f(y)$ are $S$-equivalent.
A main benefit of computable reducibility is that it provides a single formal setting for classifying countable equivalence relations, even if they arise from very different contexts.
For instance, Miller III~\cite{miller1971group} showed that there is a finitely presented group such that all computably enumerable equivalence relations are computably reducible to its word problem.
As another example --- this one concerning relations that are not even hyperarithmetical ---
Fokina, S.~Friedman, Harizanov, Knight, McCoy, and Montalb\'an~\cite{Fokina:12b} proved that the isomorphism relations on several classes of computable structures (e.g., graphs, trees,  torsion abelian groups,   fields  of  characteristic $0$ or  $p$, linear orderings) is complete  among $\Sigma^1_1$ equivalence relations.

More generally, researchers studied computable reducibility for decades, and approached it from several different perspectives, often unveiling significant connections with other fields, such as descriptive set theory and proof theory.
 Ershov~\cite{Ershov:77} initiated this research program
 in a category-theoretic fashion, while dealing with the theory of numberings
 (see~\cite{Ershov:survey} for a survey in English). Following Ershov, one can define the category of equivalence
relations on $\omega$, in which a morphism from $R$ to $S$ is a function $\mu : \omega/{_R} \rightarrow \omega/{_S}$ such that there is a computable
function $f$ with $\mu([x]_{R}) = [f(x)]_{S}$. So morphisms are induced by computable functions $f$ such that $x R y \Rightarrow f(x) S f(y)$, and thus $R\leq S$ holds if and only if there is monomorphism from $R$ to $S$.

Scholars continued Ershov's work by pursuing different goals, such as studying provable equivalence of formal systems (see, e.g.,~\cite{Visser:80,Montagna:82,bernardi1981,Bernardi:83}). This proof-theoretic motivation explains why they focused on the $\Sigma^0_1$ case: the set of theorems of any computably axiomatizable theory is obviously a computably enumerable set. In the Russian literature, c.e.\  equivalence relations are often called \emph{positive}, but
 as in~\cite{Gao:01} and many other contributions, we adopt the acronym \emph{ceers} for referring to them.
The interested reader can see Andrews, Badaev, and Sorbi~\cite{andrews2017survey} for a nice and up-to-date survey on ceers, with a special focus on \emph{universal} ceers, i.e., ceers to which all others can be computably reduced.

Computable reducibility shall also be regarded as the computable analogue of \emph{Borel reducibility}, a central object of study of modern descriptive set theory.
Introduced by H.~Friedman and Stanley~\cite{friedman1989borel}, the notion of Borel reducibility allows to compare the complexity of equivalence relations on Polish spaces, such as the Cantor space $2^\omega$ (see~\cite{gao2008invariant,kanoveui2008borel}). This is particularly meaningful for calculating the complexity of different classification problems, i.e., problems associated to the task of characterizing some collection of mathematical objects  in terms of invariants  (up to isomorphism, or some other nice equivalence relation expressing structural resemblance). Calvert, Cummins, Knight, and S.~Miller~\cite{calvert2004comparing} introduced an effective version of this study, by considering effective transformations between classes of structures. Another possible approach is that of regarding computable reducibility itself as representing a computable counterpart of Borel reducibility, where the former naturally applies to equivalence relations with domain $\omega$ and the latter refers to equivalence relations on $2^\omega$ (or similar spaces). This interpretation appears, e.g., in~\cite{Gao:01,Coskey:12,fokina2010effective,miller2016finitary}. In particular, Coskey, Hamkins, and R.~Miller~\cite{Coskey:12} investigated equivalence relations on (indices of) c.e.\ sets --- or, of families of c.e.\ sets --- that mirror  classical combinatorial equivalence relations of crucial importance for Borel theory.

\subsection{Reducibility and bi-reducibility spectra}
Our motivating question is the following:
\emph{given two arbitrary equivalence relations $R$ and $S$, how much information is needed to compute  possible ways of reducing $R$ to $S$?}
 As our main tool, we introduce the following spectra of Turing degrees, that stand in analogy with many similar notions from computable structure theory.

\begin{definition}
Let $(R, S)$ be a pair of equivalence relations. The \emph{degree spectrum of reducibility} of $(R,S)$ (or, the \emph{reducibility spectrum} of $(R,S)$) is the following set of Turing degrees
\[
\Spec_{\Rightarrow}(R, S)=\set{\mathbf{d} \mbox{  $|$ $R \leq_{\mathbf{d}} S$} }.
\]

Similarly, we define the \emph{degree spectrum of bi-reducibility} of $(R,S)$ (or, the \emph{bi-reducibility spectrum} of $(R,S)$) as
\[
\Spec_{\Leftrightarrow}(R, S)=\set{\mathbf{d} \mbox{ $|$} R\equiv_\mathbf{d} S}.
\]
\end{definition}

\begin{definition}
The \emph{degree of reducibility} of $(R,S)$ is the least degree of $\Spec_{\Rightarrow}(R, S)$ if any such degree exists. The \emph{degree of bi-reducibility} of  $(R,S)$ is defined similarly.
\end{definition}

Different degree spectra have been considered in the literature. The \emph{isomorphism spectrum} of a structure $\mathcal{A}$ (in symbols: $\Spec_{\cong}(\A)$) is classically defined as the collection of all Turing degrees that compute a copy of $\A$ and is the most common way to measure the computational complexity of $\A$.
Isomorphism spectra have been widely investigated (see, e.g.,~\cite{knight_degrees_1986, kalimullin2008almost, andrews2016complements, harizanov2007spectra}; they are also called ``degree spectra'' or ``spectra of structures'':
our terminology emphasizes the difference with the other spectra discussed below).  In recent years researchers considered also alternative degree spectra, such as \emph{theory spectra}~\cite{andrews_spectra_2015}, \emph{$\Sigma_n$-spectra}~\cite{fokina_degree_2016}, \emph{bi-embeddability spectra}~\cite{fokina_bi-embeddability_2016}, and \emph{elementary bi-embeddability spectra}~\cite{rossegger_elementary_2018}.
The notion of \emph{computable categoricity} gives rise to the \emph{categoricity spectrum}, that measures how difficult it is to compute isomorphisms between computable copies of a given structure (see~\cite{fokina2010degrees}, where the notion was introduced). Our perspective is to some extent close to the latter spectra, since we similarly fix the structures involved and then analyse the information needed to witness a possible reduction between them.

The main problem when dealing with a given class of spectra is that of characterizing which kind of information they can encode, in terms of the classes of Turing degrees that can be realised by them. Even if $\A$ is a familiar structure, $\Spec_{\cong}(\A)$ (or some other possible spectrum of $\A$) might be very complicated: for example, the isomorphism spectrum of a linear order $L$ is a \emph{cone} (i.e., $\Spec_{\cong}(L)=\set{\mathbf{d}: \mathbf{c}\leq \mathbf{d}}$, for some $\mathbf{c}$) if and only if $L$ is computably presentable, as proved by Richter~\cite{richter_degrees_1981}.

\medskip

In the present paper we aim to compare reducibility and bi-reducibility spectra with other spectra considered in the literature.

\begin{definition}
Let $\dA$ be a set of Turing degrees. A set of Turing degrees $\dB$ is a \emph{basis} of $\dA$ if
\begin{enumerate}
\item all elements of $\dB$ are Turing-incomparable,
\item and $\dA=\set{\mathbf{d}: (\exists \mathbf{b} \in \dB)(\mathbf{b}\leq \mathbf{d})}$.
\end{enumerate}
\end{definition}

 In Section $2$ we prove that any upward closed set of Turing degrees with a countable basis can be realised as a reducibility spectrum. In Section $3$ we show that there is a reducibility spectrum with no countable basis, while in Section $4$ we show that there is a reducibility spectrum having no basis at all. In Section $5$ we partially extend these results to the case of bi-reducibility spectra.

\subsection{Notation and terminology} For all $X\subseteq \omega$, we denote $\omega \smallsetminus X$ by $\overline{X}$. All our equivalence relations have domain $\omega$. We denote by $[x]_R$ the $R$-equivalence class of any given $x$. We denote by $c_R$ the number of equivalence classes of $R$.
An equivalence relation
$R$ is $n$-\emph{bounded}
if all its equivalence classes have size at most $n$; $R$ is \emph{bounded} if it is $n$-bounded for some natural number $n$. The following basic equivalence relations will appear many times:

\begin{itemize}
\item $\texttt{Id$_1$}$ is the equivalence relation consisting of just one equivalence class, i.e., $x\texttt{Id$_1$} y$, for all $x,y$;
\item $\texttt{Id}$ is the equivalence relation consisting of all singletons, i.e., $x\texttt{Id} y$ if and only if $x=y$.
\end{itemize}

Our computability theoretic notions are standard, and as in~\cite{Soare:Book}. In particular, recall how the principal function of a given infinite set $A$ is defined.
Let $a_0<a_1<\ldots$ be the ascending sequence of all elements of $A$.
\emph{The principal function} of $A$ is the following injective function: $p_A(x)=a_x$.
It is immediate that $p_A$ is computable in $A$.

\section{Classes of Turing degrees realised by reducibility spectra}
For the sake of exposition, we begin by focusing on reducibility spectra. The discussion about bi-reducibility spectra is postponed to the last section.

It is easy to see from the definition of reducibility spectra that they are either upwards closed or empty.

\begin{proposition}\label{fact:upward-density}
Let $(R,S)$ be any pair of equivalence relations. Then $\Spec_{\Rightarrow}(R, S)$ is either empty or upward closed and, if empty, then  $c_R > c_S$.
\end{proposition}

\begin{proof}
If $c_R >c_S$, then obviously there can be no reduction from $R$ into $S$, since any function would necessarily map distinct equivalence classes of $R$ into one single class of $S$.

Next, assume $c_R \leq c_S$.
Let $\mathbf{a}=\deg(R\oplus S)$, and define $f$ to be the $\mathbf{a}$-computable function such that $f(0)=0$ and

\[ f(x+1)=
\begin{cases}
f(y) &\text{ $(\exists y\leq x)(y\in [x+1]_R)$}\\
\mu z[(\forall y \leq x)(z\notin[f(y)]_S)] &\text{ otherwise}.
\end{cases}
\]

The hypothesis that $c_R \leq c_S$ holds guarantees that
any fresh equivalence class of $R$ can be mapped into a fresh  equivalence class of $S$. Thus, $f$ is well-defined and $R\leq_{\mathbf{a}}S$ via $f$. This proves that  $\Spec_{\Rightarrow}(R, S)$ is not empty.
For the upward closure just notice that if $R$ $\mathbf{a}$-computably reduces to $S$, then the reduction holds also for any $\mathbf{d}$ such that $\mathbf{d}\geq \mathbf{a}$.
\end{proof}

\begin{remark}
To avoid empty spectra, in what follows we assume that all our equivalence relations have infinitely many equivalence classes.
\end{remark}

\smallskip

Although our analysis of reducibility spectra will focus mainly on positive results, we begin with a negative one: standard set-theoretic considerations show that reducibility spectra do not coincide with the class of all upward closed collections of Turing degrees.

\begin{proposition}
There exists an upward closed collection of Turing degrees that can not be realised as a reducibility spectrum.
\end{proposition}

\begin{proof}
On the one hand, there are $2^{\aleph_0}$ many reducibility spectra. This follows immediately from the fact that any such  spectrum is associated with a pair of equivalence relations with domain $\omega$. On the other hand, there are $2^{2^{\aleph_0}}$ many upward closed collections of Turing degrees. To see this, recall that there are $2^{\aleph_0}$ minimal degrees with respect to Turing reducibility and notice that any of set of minimal degrees form the basis of an upward closed collection of degrees.
\end{proof}

Proposition~\ref{fact:upward-density} implies
that any pair $(R,S)$ has degree of reducibility
if and only if its reducibility spectrum
is a cone. The next result says that, for any Turing degree~$\mathbf{d}$, there is a pair of equivalence relations $(R,S)$ that encodes $\mathbf{d}$ as its degree of reducibility. We begin by introducing a convenient way of coding sets of numbers by equivalence relations.

\begin{definition}
Let $A_0,\ldots,A_{n-1}$  be a collection of $n$
pairwise disjoint sets of numbers.  An equivalence relation $R_{A_0,\ldots,A_{n-1}}$ is \emph{generated} by $A_0,\ldots,A_{n-1}$ if
\[
xR_{A_0,\ldots,A_{n-1}}y \Leftrightarrow x=y \vee (\exists i < n)(x,y \in A_i ).
\]
\end{definition}

This way of representing sets by equivalence relations is common in the literature (see, for instance, the characterization of \emph{set-induced} degrees of equivalence relations in Ng and Yu~\cite{ngyu}). Gao and Gerdes~\cite{Gao:01} call equivalence relations generated by $n$ sets \emph{$n$-dimensional}. We adopt their terminology when convenient.
A special reason of interest for $1$-dimensional ceers is due to the fact that the interval $[\mathbf{0_1},\mathbf{0'_1}]$ of the $1$-degrees is embeddable into the degree structure generated by computable reducibility on  ceers (computably enumerable equivalence relations). As a corollary, the first-order theory of ceers is undecidable, as is shown in~\cite{Andrews:14}.

The $1$-dimensional equivalence relations can easily encode any given Turing degree as a degree of reducibility.

\begin{proposition}\label{fact:onecone}
For any Turing degree $\mathbf{a}$, there is a pair of equivalence relations $(R,S)$, such that $\Spec_{\Rightarrow}(R,S)=\set{\mathbf{d}: \mathbf{a} \leq \mathbf{d}}$.
\end{proposition}

\begin{proof}
Let $A\subseteq \omega$ be
co-infinite and
such that $\deg(A)=\mathbf{a}$. Let $a$ be an element of $A$. Consider $(R_A, \texttt{Id})$. We define $h$ as the following $\mathbf{a}$-computable function, for all~$x$,
\[
h(x)=\begin{cases}
a &\text{ if $x\in A$},\\
x &\text{ otherwise}.
\end{cases}
\]

We have that $h$ $\mathbf{a}$-computably reduces $R_A$ to $\texttt{Id}$.
Thus,  $\mathbf{a}\in \Spec_{\Rightarrow}(R_A,\texttt{Id})$. Now, suppose that $f$ $\mathbf{d}$-reduces $R_A$ to $\texttt{Id}$. It follows that there is some $z$ such that, for all $x$, $f(x)=z$ if and only if $x \in A$. This means that $f$ computes $A$, and therefore $\mathbf{d}\geq \mathbf{a}$. So, $\Spec_{\Rightarrow}(R_A, \texttt{Id})$ is the cone above $\mathbf{a}$.
\end{proof}

The following question naturally arises: does every pair of equivalence relations have a degree of reducibility?
We answer to this question negatively. In fact, much more can be proved.

It is a well known fact that the isomorphism spectrum of a structure can not be the union of finitely many or countably many cones of Turing degrees, (see Soskov~\cite{Soskov}). This is not true for theory spectra. Andrews and Miller~\cite{andrews_spectra_2015} showed that there is a theory $T$ such that its spectrum coincides with the union of two cones. The same holds for $\Sigma_n$-spectra, with $n\geq 2$, as
proved
by Fokina, Semukhin, and Turetsky~\cite{fokina_degree_2016}.
By the following theorem, we show that reducibility spectra can be the union of $n$ cones, for all $n$.
In fact, given any countable set of Turing degrees $\dB$, there is a reducibility spectrum that coincides with the upward closure of $\dB$. A similar result holds for bi-reducibility spectra, as discussed in the final section.

\begin{thm}\label{thm:characterization}
Any upward closed collection of Turing degrees with a countable basis can be realised as a reducibility spectrum.
\end{thm}

Before proving the theorem, let us recall the notion of introreducible set.

\begin{definition}[see Jockusch~\cite{Jockusch:68}]
An infinite set $A\subseteq \omega$ is \emph{introreducible} if it is computable in any of its infinite subsets, i.e., for all infinite $B\subseteq A$, we have $B\geq_T A$.
\end{definition}

It is well known that any Turing degree $\mathbf{d}$ contains an introreducible set. To prove it, from any infinite $A\in \mathbf{d}$ build a Turing equivalent set $B$ as follows: for each $\sigma\subseteq \chi_A$, put (the code of) $\sigma$ in $B$. It is not difficult to see that from any infinite subset of $B$ we can extract arbitrarily long initial segments of $\chi_A$, and thus compute $B$. A nice consequence of the fact that $B$ consists of initial segments is that, if some function $f$ enumerates an infinite subset of $B$, then $f$ computes $B$. Since we make use of the latter property several times through the rest of the paper, it is convenient to dignify it as a lemma.

\begin{lemma}\label{lem:initial}
Let $B$ be introreducible and let $f$ be a function. If there is an infinite $Y\subseteq B$ such that $Y$ is c.e.\ in $f$, then $f\geq_T B$.
\end{lemma}

\begin{proof}
Every infinite c.e.\ set contains an infinite computable set. Relativising this, we see that there is a $Z\subseteq Y$ computable in $f$ so that $Z$ is infinite. Since $B$ is introreducible, $Z$ computes $B$ and thus $f\geq_T B$.
\end{proof}

We can now prove the theorem.

\begin{proof}[Proof of Theorem~\ref{thm:characterization}]
Let $\dA$ be an upward closed collection of Turing degrees having basis $\dB$. First, assume $\dB$ to be infinite.
For any $\mathbf{b}_i \in \dB$, denote by $B_i$ an introreducible set that belongs to $\mathbf{b}_i$. Next, let
\[
X=\set{ \langle p_{B_0}(i),p_{B_i}(x)\rangle : i,x \in \omega},
\]

where $p_{B_0}$ (resp.\ $p_{B_i}$) is the principal function of $B_0$ ($B_i$).

\smallskip

We shall think of $X$ as consisting of countably many columns such that its $i$th column encodes the set $B_i$, indexed by the $i$th element of $B_0$.
 We claim that $\Spec_{\Rightarrow}(\texttt{Id}, R_{\overline{X}})=\dA$.

\smallskip

We first show that $\dA\subseteq \Spec_{\Rightarrow}(\texttt{Id}, R_{\overline{X}})$. For $j\in \omega$,
consider the function
\[
h_j(x)=\langle p_{B_0}(j), p_{B_j}(x)\rangle,
\]

We have that, for all $j$, $\texttt{Id}$ is $\mathbf{b}_j$-computably reducible to $R_{\overline{X}}$ via $h_j$, since $h_j$ injectively maps singletons of $\texttt{Id}$ into the $j$th column of $X$, and therefore into singletons of $R_{\overline{X}}$. Thus $\dB\subseteq \Spec_{\Rightarrow}(\texttt{Id}, R_{\overline{X}})$ and then, since $\Spec_{\Rightarrow}(\texttt{Id}, R_{\overline{X}})$ is upward closed, $\dA \subseteq \Spec_{\Rightarrow}(\texttt{Id}, R_{\overline{X}})$.

Now we have to prove that  $\Spec_{\Rightarrow}(\texttt{Id}, R_{\overline{X}})\subseteq \dA$. Suppose that $\texttt{Id}\leq_{\mathbf{d}}R_{\overline{X}}$ via some $f$. We want to show that there is $k$ such that $f$ computes $B_k$, hence witnessing that $\mathbf{d}\in \dA$. Notice that the range of $f$ is all contained, with the exception of at most one element, in $X$. Consider two cases.

\begin{enumerate}
\item Suppose that there is $k$ such that $f$ enumerates an infinite subset $Y$ of the $k$th column of $X$. If so, then the set $\set{p_{B_k}(x) : \langle p_{B_0}(k),p_{B_k}(x)\rangle \in Y}$ is an infinite subset of $B_k$ which is c.e.\ in $f$.
By Lemma~\ref{lem:initial}, this means that $f$ computes $B_k$, and so $\mathbf{
d} \geq \mathbf{b}_k$.
\item If $f$ enumerates only finitely many elements for each column, then it must be the case that $f$ picks infinitely many columns, i.e., the set
\[
Y=\set{p_{B_0}(k) : (\exists y)(\langle p_{B_0}(k),y\rangle \in \range(f))}
\]
must be infinite. Since $Y\subseteq B_0$ and $B_0$ is introreducible, by Lemma~\ref{lem:initial} we obtain that $\mathbf{d}\geq \mathbf{b}_0$.
\end{enumerate}
Thus, we have that $\Spec_{\Rightarrow}(\texttt{Id}, R_{\overline{X}})\subseteq \dA$. So we conclude that $\Spec_{\Rightarrow}(\texttt{Id}, R_{\overline{X}})=$~$\dA$.

\smallskip

If $\dB=\set{\mathbf{b}_0,\ldots,\mathbf{b}_n}$ is finite the proof is essentially the same. One can apply the above construction as follows: when $X$ is to be constructed, use the first $n+1$ columns of $X$ to encode $B_i\in \mathbf{b}_i$, for $i \leq n$, and then encode $B_0$ into all remaining columns.
\end{proof}

\section{Reducibility spectra that contain $\Pi^0_1$-classes}
In the previous section, we proved that reducibility spectra are rather expressive: any countable collection of cones can be realised as a reducibility spectrum. In this section, we go one step further. We show that there is a reducibility spectrum that contains a special $\Pi^0_1$-class, i.e., one with no computable member. %Moreover, we show that such a spectrum is a spectrum of ceers.
In doing so, we continue our analysis  of reducibility spectra of the form  $\Spec_{\Rightarrow}(\texttt{Id}, R)$,  focusing this time on the case when $R$ is a ceer.

The next definition appears in Andrews and Sorbi~\cite{andrewssorbi2}.

\begin{definition}[Andrews and Sorbi~\cite{andrewssorbi2}]
A ceer $R$ is \emph{light} if $\texttt{Id} \leq R$. Otherwise, if $R$ has infinitely many equivalence classes, it is \emph{dark}.
\end{definition}

Dark ceers exist. As an easy example, consider a $1$-dimensional ceer $R_S$ where $S$ is a simple set. If $\texttt{Id}\leq R_S$ via some computable $f$, then we would have that $\range(f)\smallsetminus S$ is an infinite c.e.\ subset of $\overline{S}$, contradicting the fact that the latter is immune. More generally, the distinction between light and dark ceers reflects a fundamental distinction about how much information we can effectively extract from a given ceer: light ceers are  those for which there exists some computable listing $(y_i)_{i\in\omega}$ of pairwise nonequivalent numbers, while for dark ceers no such listing is possible (see~\cite{andrewssorbi2} for an extensive study of light and dark ceers and how they behave with respect to
the existence of joins and meets of ceers).

For our present interests, it is immediate to observe that $R$ is light if and only if $\Spec_\Rightarrow(\texttt{Id},R)$  is the cone above $\mathbf{0}$. Hence, we shall turn to dark ceers for nontrivial spectra. In particular, we want to investigate how complicated $\Spec_\Rightarrow(\texttt{Id},R)$ can be if $R$ is dark. Our strategy is to consider the class of all partial transversals of $R$.

\begin{definition}
Let $R$ be an equivalence relation. A set $A\subseteq \omega$ is a \emph{partial transversal} of $R$ if  $x,y\in A$ implies
$\neg(xRy)$,  for $x\neq y$. Denote by $P(R)$ the class of all partial transversals of $R$.
\end{definition}

Let us stress that we think of transversals of $R$ as sets and not functions. Hence, according to the last definition, $P(R)$ is a subset of Cantor space instead of a subset of Baire space. Our choice is motivated by the fact that we want to make use of Theorem \ref{thm:treeminimal} below, which holds for computably bounded $\Pi^0_1$-classes.

\begin{proposition}\label{prop:tranversaldegree} For any equivalence relation $R$,
\[
\mathbf{d}\in \Spec_{\Rightarrow}(\emph{\texttt{Id}},R) \mbox{ if and only if $\mathbf{d}$ computes some $A\in P(R)$}.
\]
\end{proposition}

\begin{proof}
$(\Rightarrow)$: If $\mathbf{d}\in \Spec_{\Rightarrow}(\texttt{Id},R)$, then there is a $\mathbf{d}$-computable function $f$ such that $\range(f)\in P(R)$.

$(\Leftarrow)$: If $A$ is a $\mathbf{d}$-computable infinite partial transversal of $R$, then $\texttt{Id} \leq R$ via $p_A$, and therefore $\mathbf{d}\in \Spec_{\Rightarrow}(\texttt{Id},R)$.
\end{proof}

The transversals of a given ceer $R$ obviously form a $\Pi^0_1$-class of functions. Similarly, $P(R)$ forms a $\Pi^0_1$-class of sets.

\begin{lemma}\label{lem:tree}
If $R$ is a ceer, then $P(R)$ is a $\Pi^0_1$-class.
\end{lemma}

\begin{proof}
We construct a binary computable tree $T_R$ such that $[T_R]$ (i.e., the infinite branches through $T_R$)  coincides with $P(R)$. The idea is to freeze a node $v$ of $T_R$ whenever we witness that a branch passing through it fails to encode a partial transversal of $R$. This happens if the path from the root to $v$  picks numbers from the same equivalence class.

More formally, for any $\sigma \in 2^{<\omega}$ of length $s$, let $\sigma$ be in $T_R$ if and only if the following holds
\[
(\forall x,y\leq s) [(x\neq y \wedge \sigma(x)=\sigma(y)=1) \Rightarrow \neg(x R_s y)].
\]

$T_R$ so defined is obviously a computable tree and, from the definition, it is clear that $[T_R]$ coincides with $P(R)$.
\end{proof}

Our goal is now to apply  the following classical theorem due to Jockusch and Soare~\cite{jockusch1972} to the case of reducibility spectra.

\begin{thm}[Jockusch and Soare~\cite{jockusch1972}]\label{thm:treeminimal}
Let $\mathcal{C}$ be a special $\Pi^0_1$-class. $\mathcal{C}$ contains $2^{\aleph_0}$ elements any two of which form a minimal pair.
\end{thm}

The main obstruction is that $P(R)$ include also finite transversals. This means that, for any given $R$, we have that $\mathbf{0}\in \set{\deg(A) : A \in P(R)}$, which makes the latter set useless for the analysis of $\Spec_\Rightarrow(\texttt{Id},R)$.
To overcome this problem, we
consider ceers of the following kind.

\begin{lemma}\label{prop:darknonhyper}
There is a dark ceer $R$ that has no infinite equivalence classes and such that $P(R)$ contains an infinite nonhyperimmune element.
\end{lemma}

\begin{proof}
We construct $R$ by stages, i.e., $R=\bigcup_{s\in \omega} R_s$. We design the construction to meet the following requirements in such a way that  $P(R)$ contains an infinite nonhyperimmune element
\[
\mathcal{P}_e : \mbox{ if $W_e$ is infinite, then $W_e \notin P(R)$},
\]
\[
\mathcal{N}_e: \mbox{ $[e]_R$ is finite}.
\]

The priority ranking of the requirements is the following
\[
\mathcal{P}_0 > \mathcal{N}_0 > \dots > \mathcal{P}_e > \mathcal{N}_e > \cdots
\]

Notice that if all $\mathcal{P}$-requirements are  met, then $R$ is necessarily dark. Otherwise, there would be an injective computable function $f$ such that $\range(f)\in P(R)$, and this would contradict any requirement $\mathcal{P}_e$ with $W_e=\range(f)$. On the other hand, if all $\mathcal{N}$-requirements are met, then all equivalence classes of $R$ are obviously finite.

\subsection*{Construction} Let us set some terminology.
We \emph{collapse} two equivalence classes $[x]_R$ and $[y]_R$ by adding into $R$  the pairs needed to obtain $[x]_R=[y]_R$. At any given stage, a $\mathcal{P}$-requirement is either in  \emph{stand-by} or  \emph{settled}. Moreover, if some action designed to deal with a given requirement $\mathcal{P}_e$ has the effect of expanding $[i]_R$, then we say that $\mathcal{P}_e$  \emph{disturbs} $\mathcal{N}_i$.

\medskip
\noindent \emph{Stage $0$}: $R=\set{(x,x) : x \in \omega}$. Put all $\mathcal{P}$-requirements in stand-by.

\smallskip

\noindent \emph{Stage $s+1=\langle e, n\rangle$}: Deal with $\mathcal{P}_e$.
If $\mathcal{P}_e$ is in stand-by and there are $x,y\in W_{e,s}$ with $\min([x]_{R_s}\cup [y]_{R_s})\geq3e$, then do the following:

\begin{enumerate}
\item Collapse $[x]_{R_s}$ and $[y]_{R_s}$ and call $(x,y)$ the pair of \emph{witnesses} of $\mathcal{P}_e$;
\item Set $\mathcal{P}_e$ as settled.
\end{enumerate}

Otherwise, do nothing.

\subsection*{Verification} The verification is based on the following claims.

\begin{claim}\label{claim:nrequi}
All $\mathcal{N}$-requirements are satisfied, i.e, $R$ has no infinite equivalence classes.
\end{claim}

\begin{proof}
Towards a contradiction, assume that there is a least requirement $\mathcal{N}_e$ that is not satisfied. It follows from the construction that any $\mathcal{N}$-requirement can be disturbed only by $\mathcal{P}$-requirements with higher priority (in fact, it is immediate to see that if $\mathcal{P}_e$ disturbs $\mathcal{N}_i$ then $e\leq \left \lfloor{\frac{i}{3}}\right \rfloor$). Let $s$ be a stage such that no $\mathcal{P}$-requirement with priority higher than $\mathcal{N}_e$ acts after $s$. Such $s$ must exist since each $\mathcal{P}$-requirement acts at most once. We have that $[e]_R$ will never be expanded after stage $s$, since no $\mathcal{P}$-requirement with lower priority is allowed to disturb it, and thus eventually remains finite.
\end{proof}

An immediate consequence of the last claim is that $R$ has infinitely many classes.

\begin{claim}
All $\mathcal{P}$-requirements are satisfied, i.e, $R$ is dark.
\end{claim}

\begin{proof}
Towards a contradiction, assume that there is a least requirement $\mathcal{P}_e$ that is not satisfied. By Claim~\ref{claim:nrequi}, we know that there must be a stage $s$ such that all $\mathcal{N}$-requirements with  priority higher than $\mathcal{P}_e$ are never disturbed after $s$. This means that there exists a least $k$ such that $\bigcup_{i<e}[i]_R\subseteq \set{x : 0 \leq x < k}$. Therefore, since $W_e$ is infinite and $R$ has infinitely many classes, we have that there is a stage $t+1=\langle e,n\rangle>s$ such that $x,y \in W_{e,t}$ and $\min([x]_{R_t}\cup [y]_{R_t})>\max \set{k,3e}$.
So, at stage $t$ the construction collapses $[x]_{R_t}$ and $[y]_{R_t}$. But this action excludes $W_e$ from the partial transversals of $R$. That is to say, we obtain $W_e\notin P(R)$, which contradicts the initial hypothesis that $\mathcal{P}_e$ was not satisfied.
\end{proof}

\begin{claim}
$P(R)$ contains an infinite nonhyperimmune member.
\end{claim}

\begin{proof}
Let $Z$ be the set of numbers that are not witnesses of any $\mathcal{P}$-requirement. %It is not difficult to see that t
These elements correspond to the singletons of $R$, as a given equivalence class $[x]_R$ has size bigger than $1$ if and only if there is a $\mathcal{P}$-requirement that picks $x$ as a witness. Furthermore, notice that for all $k$ the set $\set{x : 0\leq x \leq 3k}$ contains at least $k$ singletons of $R$. This is because, by construction, if $\mathcal{P}_e$ chooses a pair of witnesses $(x,y)$, then $\min(x,y)\geq 3e$. We use this fact to build a strong array containing an infinite partial transversal as follows. Without loss of generality, assume that $|[0]_R|=1$. Let $f$ be a  computable function defined by recursion
\[
D_{f(0)}=\set{0},
\]
\[
D_{f(i+1)}=\set{x: \max(D_{f(i)})<x\leq 3(\max(D_{f(i)})+1)}.
\]

\medskip

It is immediate to notice that the so defined sets are pairwise disjoint.
Recall that for all $k$ the set $\set{x : 0\leq x\leq 3(k+1)}$ contains at least $k+1$ singletons of $R$. From this fact, we obtain that, for all $i$, there exists $y$ such that $|[y]_R|=1$ and
\begin{equation}
y \in D_{f(i+1)}=(\set{x : 0\leq x\leq 3(\max(D_{f(i)}+1))}\smallsetminus \set{0\leq x\leq \max(D_{f(i)})}).
\end{equation}

Consider now the set $A=\set{x : |[x]_R|=1 \wedge (\exists i)(x \in D_{f(i)})}$. It is obvious that $A\in P(R)$, because $A$ contains only singletons of $R$.  From $(1)$ above, it follows that $A \cap D_{f(i)} \neq \emptyset$  for all $i$. Thus, $A\in P(R)$ is infinite and nonhyperimmune.
\end{proof}
\end{proof}

\begin{thm}\label{thm:nocount}
There is a reducibility spectrum (not containing $\mathbf{0}$) that contains a special $\Pi^0_1$-class.
\end{thm}

\begin{proof}
Let $R$ be as in Lemma~\ref{prop:darknonhyper} and let $A\in P(R)$ be an infinite nonhyperimmune set.  Let $\set{D_{f(i)}}_{i \in \omega}$ be a strong array witnessing the nonhyperimmunity of  $A$. From $R$ we construct the tree $T_R$ as in the proof of Lemma~\ref{lem:tree}. Next, we computably build a subtree $T$ of $T_R$ by allowing in $T$ only the elements of $P(R)$ that meet every $D_{f(i)}$. That is to say,  we freeze a node $v$ of $T$ if we see that any path passing through $x$ fail to  intersect some $D_{f(i)}$.
More formally, for any $\sigma \in 2^{<\omega}$ of length $s$, let $\sigma$ be in $T$ if and only $\sigma \in T_\R$ and
\[
(\forall i\leq s)(\exists x) (x \in D_{f(i)} \wedge \sigma(x)=1).
\]

$T$ is obviously computable, because $T_\R$ is computable and in order to establish whether some $\sigma\in T$ we have to check only finitely many finite sets.
Moreover $[T]$ is special, i.e., it has no computable member. This follows from the following facts: $T$ is a subtree of $T_\R$; all infinite elements of $P(R)$ are immune (since $R$ is dark); and
any member of $[T]$ is (the characteristic function of) an infinite set, since no finite set can intersect $\set{D_{f(n)}}_{n\geq 0}$ infinitely many times.
%
%
%Towards a contradiction suppose that $\Spec_{\Rightarrow}(\texttt{Id},\R)$ has a countable basis $\dB$. By Proposition~\ref{prop:tranversaldegree} and the fact that $T$ is a subtree of $T_R$, we obtain that
%\[
%\set{\deg(A): A \in [T]}\subseteq \Spec_{\Rightarrow}(\texttt{Id},\R).
%\]
%
%But $[T]$, being special, has continuum many members that pairwise form a minimal pair (Theorem~\ref{thm:treeminimal}).  Therefore, if $\dB$ is countable, there must be a minimal pair which lies above some $\mathbf{b}\in \dB$. It follows that $\mathbf{b}=\mathbf{0}$. So $\Spec_\Rightarrow(\texttt{Id},R)$ coincides with the cone above $\mathbf{0}$, and therefore $R$ turns out to be light, but this contradicts our initial choice of $R$ as dark.
\end{proof}

%As a corollary, we immediately obtain that there is a reducibility spectrum with no countable basis. In this section we prove something stronger: there are reducibility spectra with no basis at all.
Now, let $R$ be the ceer constructed in the proof of Theorem~\ref{thm:nocount} and assume that $\Spec_\Rightarrow(\Id,R)$ has a countable basis. Then, as by Theorem~\ref{thm:treeminimal} $\Spec_\Rightarrow(\Id,R)$ contains continuum many minimal pairs there must be an element of the basis Turing below a minimal pair in $\Spec_\Rightarrow(\Id,R)$ and thus this element must be $\mathbf 0$ contradicting that $R$ is dark. We have just proven the following.
\begin{corollary}\label{cor:nocount}
  There is a reducibility spectrum with no countable basis.
\end{corollary}
In the next section we will prove an even stronger result; that there are reducibility spectra with no basis at all.

To conclude the special focus on spectra of the form $\Spec_\Rightarrow(\Id,R)$,
it is worth noticing that in the above proofs we never used the fact that all equivalence classes of $R$ are finite. In fact, it can be shown that the same result holds also for (properly constructed) equivalence relations with infinite equivalence classes. Yet, our choice makes the result sharp, in the sense of the following proposition.

\begin{proposition}
If $R$ is a bounded ceer, then $\Spec_\Rightarrow(\emph{\texttt{Id}},R)$ has  a countable basis; in fact,  $\mathbf{0}\in \Spec_\Rightarrow(\emph{\texttt{Id}},R)$.
\end{proposition}

\begin{proof}
Let $k$ be the largest number such that $R$ has infinitely many equivalence classes of size $k$. Let $Y=\set{y : |[y]_R|>k}$. $Y$ is obviously finite. We can then easily construct a c.e.\ partial transversal $A$ of $R$: when we witness that $|[x]_R|=k$, for some $x\notin Y$, we put the least element of $[x]_R$ into $A$.
Then, $\texttt{Id} \leq_{\mathbf{0}} R$ via any computable function $f$ with $\range(f)=A$.
\end{proof}

\section{A reducibility spectrum with no basis}
In this section we complete the picture about the complexity of reducibility spectra. Having shown that these spectra can be with no countable basis, it is natural to ask whether they all have a basis.  Notice that the question has not been already answered by Corollary~\ref{cor:nocount}, since the spectrum considered in the proof might have a basis which is uncountable. In this section, we directly construct a reducibility spectrum with no basis at all. As in the previous section, we prove that the result already holds for ceers.
The idea for building the desired spectrum  is to make it downward dense while at the same time excluding $\mathbf{0}$ from it. More precisely, we aim to build a pair of ceers $(R,S)$ such that

\smallskip

\begin{enumerate}
\item $\mathbf{0} \notin \Spec_\Rightarrow(R,S)$,
\item if $\mathbf{d} \in \Spec_\Rightarrow(R,S)$, then  $\set{ \mathbf{c}:\mathbf{0} <\mathbf{c}< \mathbf{d}} \cap \Spec_\Rightarrow(R,S) \neq \emptyset$.
\end{enumerate}

\smallskip

As is clear, a spectrum that satisfies $(1)$ and $(2)$ can not have a basis.
The construction of $R$ and $S$ is based on the following notion due to  Cleave~\cite{Cleave:61}.

\begin{definition}[Cleave~\cite{Cleave:61}]
A sequence of pairwise disjoint c.e.\ sets $E_0\ldots,E_{n-1}$ is \emph{creative} if there is a computable function $p$ such that, if $W_{i_0},\ldots,W_{i_{n-1}}$ is a sequence of $n$ pairwise disjoint c.e.\ sets such that $W_{i_k} \cap E_{k}= \emptyset$, for all $ k<n$, then
\[
p(i_0,i_1,\ldots,i_{n-1})\in \overline{\bigcup_{0\leq k< n}{E_k \cup W_{i_k}}}.
\]
\end{definition}

Creative sequences generalise effective inseparability from pairs of c.e.\ sets to sequences of them. The underlying intuition of the latter definition is that no set of a creative sequence can be effectively separated from the others.

Effective inseparability plays a crucial role for the theory of ceers. For example, Andrews, Lempp, J.~Miller, Ng, San Mauro, and Sorbi~\cite{Andrews:14} showed that \emph{uniformly effectively inseparable} ceers
(i.e., ceers whose equivalence classes are pairwise effective inseparable, and such that this effective inseparability can be witnessed uniformly)
 are universal. This fact  subsumes all known results of universality for ceers and stands in analogy with the following  classical results: any pair of effectively inseparable sets is $1$-complete, as proven by Smullyan~\cite{smullyan1961theory}, and any creative
sequence is $1$-complete, as proved by Cleave~\cite{Cleave:61}. Cleave's result, in particular, means that, if $E_0,\ldots, E_{n-1}$ is a creative sequence, then for all sequences of pairwise disjoint c.e.\ sets $G_0,\ldots,G_{m-1}$  with $m\leq n$, there is an injective computable function $h$ such that
\[
\mbox{if }x \in G_k \mbox{ for some $k< n$}, \mbox{ then }  h(x) \in E_k;
\]
\[
\mbox{if } x \notin \bigcup_{k<n} G_k, \mbox{ then }  h(x) \notin \bigcup_{k < n} E_k.
\]

As an immediate corollary of the last fact we obtain the following.

\begin{proposition}\label{prop:univ-crea}
For all $n$, there is a $n$-dimensional ceer $R_{E_0,\ldots,E_{n-1}}$ such that, for any $m$-dimensional ceer $R_{G_0,\ldots,G_{m-1}}$ with $m\leq n$, $R_{G_0,\ldots,G_{m-1}}\leq_{\mathbf{0}} R_{E_0,\ldots,E_{n-1}}$.
\end{proposition}

\begin{proof}
Just choose the sequence $E_0,\ldots,E_{n-1}$ to be creative. Cleave's result guarantees that, given any sequence ${G_0,\ldots,G_{m-1}}$ with $m\leq n$, there is a $1$-reduction $h$ from ${G_0,\ldots,G_{m-1}}$ into $E_0,\ldots,E_{n-1}$. Since $h$ is injective, $h$ will also induce a computable reduction from $R_{G_0,\ldots,G_{m-1}}$ into $R_{E_0,\ldots,E_{n-1}}$.
\end{proof}

From now on, we denote by $\set{U_i}_{i\in \omega}$ a uniformly c.e.\ sequence  of ceers such that, for all $i$, $U_i$ is generated by a creative sequence of length $i+1$. As an example of such a sequence the reader can think of each $U_i$ as $R_{K_0,\ldots,K_i}$, where $K_j=\set{x: \phi_x(x)\downarrow=j}$ for $j \leq i$.

\smallskip

For the next lemma, recall that $B$ and $C$ \emph{split} $A$ (notation: $A= B\sqcup C$) if $B\cap C=\emptyset$ and $A=B\cup C$.

\begin{lemma}\label{lem:descend}
Let $A$ be a noncomputable c.e.\ set. There are computable functions $f,g$ such that $W_{f(0)}=A$, $W_{g(0)}=\emptyset$, and,  for all $i$,
\begin{enumerate}
\item  $W_{f(i+1)}\sqcup W_{g(i+1)}=W_{f(i)}$,
\item and $W_{f(i)}>_T W_{f(i+1)}>_T \mathbf{0}$.
\end{enumerate}
\end{lemma}

\begin{proof}
Given any noncomputable c.e.\ set $A$,  Sacks Splitting theorem~\cite[Theorem 7.5.1.]{Soare:Book} allows to construct a splitting $B \sqcup C$ of $A$ into noncomputable c.e.\ sets such that both $B$ and $C$ avoid the cone above $A$. It is enough to apply the theorem infinitely many times to obtain $f$ and $g$. Indeed, suppose we have defined $W_{f(i)}$ and $W_{g(i)}$. By applying Sacks' construction to $W_{f(i)}$, one obtains a splitting $B_i\sqcup C_i$ of $W_{f(i)}$. Then, we define $f(i+1)$ and $g(i+1)$ equal (respectively) to an index of $B_i$ and an index of $C_i$, where these indices are uniformly obtained by the $s$-$m$-$n$ theorem.
\end{proof}

\begin{definition}
The \emph{cylinder}  of a (possibly infinite) family $\mathcal{F}=\set{A_0,A_1,\ldots}$ of sets is the equivalence relation $R$ defined by
\[
\langle i, x\rangle R \langle j, y\rangle \Leftrightarrow i=j \wedge [ x=y  \vee (i < |\mathcal{F}| \wedge x,y \in A_i)].
\]
\end{definition}

\begin{thm}
There is a reducibility spectrum of ceers with no basis.
\end{thm}

\begin{proof}

We build equivalence relations $R$ and $S$ in columns, in such a way that $\Spec_\Rightarrow(R,S)$ has no basis. To do so, let $B$ be a noncomputable co-c.e.\ introreducibile set with $0\notin B$ (Lachlan proved that such $B$ must be hyperimmune, see the addendum at the end of Jockusch~\cite{Jockusch:68}). We define $R$ as the cylinder of the family $\set{W_{f(i)}}_{i\in \omega}$, where $f$ is  as in Lemma~\ref{lem:descend} and $W_{f(0)}=\overline{B}$, i.e.,
\[
\langle i, x\rangle R \langle j, y\rangle \Leftrightarrow i=j \wedge (x=y  \vee x,y \in W_{f(i)}).
\]

 We define $S$ as follows. We keep the $0$th column of $S$ isomorphic to $\texttt{Id}$, and then we encode the cylinder of $\set{U_i}_{i\geq 0}$ into the remaining columns of $S$,
with the further condition of making its $i$th column isomorphic to $\texttt{Id$_1$}$ if we witness that $i$ enters in $\overline{B}$. More formally, $S$ is the following equivalence relation
\[
\langle i, x\rangle S  \langle j, y\rangle \Leftrightarrow \langle i, x \rangle =\langle j,y\rangle \vee  (i=j \wedge i>0 \wedge (x U_{i-1} y \vee i \in \overline{B})).
\]

The relations $R$ and $S$ defined in this way are obviously ceers.
Denote by $\dA$ the upward closure of  $\set{\deg(W_{f(i)}): i \in \omega}$.
We claim that $\Spec_{\Rightarrow}(R,S)=\dA$.

\smallskip

On the one hand, let  $\mathbf{d} \in \dA$ and let $n$ be the least number such that $\mathbf{d}\geq \deg(W_{f(n)})$. By construction, any column of $S$ is either finite dimensional or isomorphic to $\texttt{Id$_1$}$. Moreover, since $B$ is infinite, there are infinitely many columns of $S$ that we never collapse to $\texttt{Id$_1$}$. Let $k$ be the least number such that the $k$th column of $S$ is $m$-dimensional with $m>n$. Denote by $C$ the cylinder of the family  ${W_{f(0)},\ldots,W_{f(n)}}$. Since  $C$ is $n+1$-dimensional, there must be a function $r$ that computably reduces $C$ to $U_m$. Indeed, by  Proposition~\ref{prop:univ-crea}, the latter is universal with respect to all  ceers of dimension $\leq m$.  Consider now the following function

\[
p(\langle i,x\rangle)= \begin{cases}
\langle k, r(\langle i, x \rangle)\rangle &\mbox{ if $i <n$}, \\
\langle 0, \langle i, 0\rangle \rangle &\mbox{ if $i \geq n$ and $x \in W_{f(i)}$}, \\
\langle 0,\langle i, x+1\rangle \rangle &\mbox{ if $i \geq n$ and $x \notin W_{f(i)}$}.
\end{cases}
\]

We want to show that $R\leq_{\mathbf{d}}S$ via $p$. First, notice that $p$ is $\mathbf{d}$-computable. Indeed, the uniformity of Sacks' Splitting guarantees that, for all $m$, $W_{f(m+1)}$ is uniformly reducible to $W_{f(m)}$. Hence,  $\mathbf{d}$ can compute any $W_{f(i)}$ with $i\geq n$.
Moreover, it is not difficult to see that $p$ reduces $R$ into $S$, by mapping the first $n+1$ columns of $R$ into the $k$th column of $S$ and all remaining columns of $R$ into singletons of the $0$th column of $S$ (the latter being isomorphic to \texttt{Id}).
This proves that $\dA \subseteq \Spec_{\Rightarrow}(R,S)$.

\smallskip
For the other inclusion, assume that $R\leq_{\mathbf{d}} S$ via some $h$. We distinguish three cases:

\begin{enumerate}
\item $h$ maps a noncomputable equivalence class of $R$ into a singleton of $S$, i.e., there exists $z$ such that, for some $k$, $h^{-1}(z)=\set{\langle k, y \rangle: y \in W_{f(k)}}$. If so, we obtain that $h$ computes $W_{f(k)}$, i.e., $\mathbf{d} \geq \deg(W_{f(k)})$;
\item $h$ maps a noncomputable equivalence class of $R$ into a collapsed column of $S$, i.e., there exists $k$ such that, for some $b \in \overline{B}$, $\set{\langle k, y \rangle: y \in W_{f(k)}}\subseteq \set{\langle b, x \rangle : x \in \omega } $. If so, since the $b$th column of $S$ is isomorphic to $\texttt{Id$_1$}$, we obtain again that $h$ computes $W_{f(k)}$, i.e., $\mathbf{d} \geq \deg(W_{f(k)})$;
\item  $h$ maps all noncomputable equivalence classes of $R$ into noncomputable equivalence classes of $S$. If so, we claim that $h$ enumerates an infinite subset of $B$: choose in a c.e.\ way a witness from each noncomputable equivalence class of $R$ and then list the indices of the column into which $h$ map such witnesses. More formally,  let $(y_i)_{i \in \omega}$ be an infinite c.e.\ sequence of numbers such that, for all $i$, $y_i\in W_{f(i)}$, and let $Y=\set {\langle k, y_k\rangle: k \in \omega}$. Notice that the set $Y^*=\set{j : (\exists x)(\langle j, x\rangle \in h[Y])}$ must be infinite. Otherwise, since each column of $S$ is finitely dimensional, $h$ would map some noncomputable equivalence class of $R$ into a singleton of $S$, and therefore we would be in Case $(1)$. Moreover, $Y^*\subseteq B$. Otherwise, $h$ would map some noncomputable equivalence class of $R$ into a collapsed column of $S$, and therefore we would be in Case $(2)$. Thus, $Y^*$ is an infinite subset of $B$ which is c.e.\ in $h$.  Since $B$ is introreducible, by Lemma~\ref{lem:initial} we obtain  $h\geq_T B$. This proves that    $\mathbf{d} \geq \deg(B)=\deg(W_{f(0)})$.
\end{enumerate}

In all the three cases $\mathbf{d}\in \set{\deg(W_{f(i)}): i \in \omega}$. Therefore,  $ \Spec_{\Rightarrow}(R,S) \subseteq$~$\dA$.

\smallskip

Hence, we proved that $\Spec_\Rightarrow(R,S)=\dA$. To conclude that   $\Spec_\Rightarrow(R,S)$ has no basis is enough to recall that  $\deg(W_{f(0)})>\deg(W_{f(1)})>\ldots$ is an infinite descending sequence of c.e.\ degrees not containing $\mathbf{0}$.
\end{proof}

\section{Bi-reducibility spectra}

We now turn our attention to bi-reducibility spectra.
By definition, any bi-reducibility spectrum is the intersection of two reducibility spectra. Indeed, for any pair of equivalence relations $(R,S)$, the following holds
\[
\Spec_{\Leftrightarrow}(R, S)=\Spec_{\Rightarrow}(R,S) \cap \Spec_{\Rightarrow}(S, R).
\]

It follows immediately that any bi-reducibility spectrum of equivalence relations with infinitely many equivalence classes is upward closed.
It is not difficult to see that all Turing degrees are degrees of bi-reducibility. But in fact, much more is true. In this section we obtain a natural companion of Theorem~\ref{thm:characterization} for bi-reducibility spectra, by proving that the latter realise any upward closed collection of Turing degrees with a countable basis. Moreover, we show that the result still holds if we limit our attention to equivalence relations with no infinite equivalence classes.

Bi-reducibility spectra are harder to deal with than reducibility spectra. This is because, while encoding or forbidding a given reduction, one has also to control backwards reductions.
This explains why the next proof is more delicate than that of Theorem~\ref{thm:characterization}.

\begin{thm}\label{thm:bi-spectra}
Let $\dA$ be an upward closed collection of Turing degrees  with countable basis $\dB$. There is a pair of  equivalence relations  $(R,S)$ with no infinite equivalence classes such that $\Spec_{\Leftrightarrow}(R,S)=\dA$.
\end{thm}

\begin{proof}
We prove the theorem for the case in which the basis $\dB=\set{\mathbf{b}_0,\mathbf{b}_1,\ldots}$ is infinite, the finite case being a simpler variation of the following argument. For any $\mathbf{b}_i\in \dB$, let $B_i\in \mathbf{b}_i$ be an  introreducible set such that $\set{0,1}\cap B_i=\emptyset$. Our strategy is to use $B_0$ and $B_1$ to encode the information provided by the other introreducible sets. In doing so, it is convenient to introduce some notation. First, similarly to the definition of $X$ in the proof of Theorem~\ref{thm:characterization}, denote by $f_i$ the following $\mathbf{b}_i$-computable function, for all $i$,
\[
f_i(x)=\langle p_{B_0}(i), p_{B_i}(x)\rangle.
\]

It is immediate to see that  the ranges of the functions so defined are pairwise disjoint.
Secondly, if $X$ is a finite set  with canonical index $z$ (i.e., $X=D_z$), denote $\langle 0, p_{B_1}(z)\rangle$ by $\ulcorner X \urcorner$. In what follows, we will use the following  observation several times: for any finite set $X$, $\ulcorner X \urcorner \notin (\bigcup_{k \in \omega} \range(f_k))$.

\smallskip

To define the desired pair of equivalence relations $(R,S)$, we start by considering the following sequence of families of finite sets
\[
\mathcal{C}_0=\set{\set{\langle 0,0\rangle,\langle 0,1\rangle}},
\]
\[
\mathcal{C}_{n+1}=\set{ f_i[X] \cup \set{\ulcorner X \urcorner}: X \in \mathcal{C}_n \wedge i\in\omega},%g \in \set{f_0,f_1, \ldots}}, %\cup \set{ f[X] \cup \set{f({\ulcorner X \urcorner})}: X \in \mathcal{C}_n}.
\]

\medskip

Let $\mathcal{C}=\bigcup_{k\in\omega}\mathcal{C}_{k}$.

\begin{claim}
$\mathcal{C}$ satisfies the following properties.

\begin{enumerate}
\item If $X \in \mathcal{C}_{n}$, then $|X|=n+1$.
\item If $X \in \mathcal{C}_n$ and $Y \in \mathcal{C}_m$, then  $X \cap Y=\emptyset$.
\end{enumerate}
\end{claim}

\begin{proof}
$(1)$ By induction we prove that, for all $n$, any two elements of $\mathcal{C}_n$ have the same size. This is trivially true for $\C_0$. Towards a contradiction, let $n$ be the least number such that there exists $\set{X,Y} \subseteq \C_n$ with $|X|\neq |Y|$. By construction, there must be $\set{X_0,Y_0}\subseteq \C_{n-1}$ such that $X=\set{f_i[X_0] \cup \set{\ulcorner X_0 \urcorner}}$, for some $f_i$, and $Y =\set{f_j[Y_0] \cup \set{\ulcorner Y_0 \urcorner}} $, for some $f_j$. Since $f_i$ and $f_j$ are both injective and  $n$ is chosen to be minimal, we have that $|f_i[X_0]|=|f_j[Y_0]|$. It follows that $|X_0|\leq |X| \leq |X_0|+1$, and $|X|<|X_0|+1$ can hold only if there is $z\in X_0$ such that $\ulcorner X_0 \urcorner= f_i(z)$. But the latter equality is impossible, since $\ulcorner X_0 \urcorner\notin (\bigcup_{k\in\omega}\range(f_k))$. Therefore, we have $|X|=|X_0|+1$. By  reasoning in a similar way, it can be shown that $|Y|=|Y_0|+1$. So, any two elements of $\C_{n}$ have the same size, and in fact they all have size $|A|+1$, for all $A \in \C_{n-1}$.

\smallskip

$(2)$
 Towards a contradiction, assume that $(n,m)$ is the least  pair for which there exist  $X\in \C_n$ and $Y \in \C_m$ such that $z \in X \cap Y$, for some $z$.
 Indeed, from the fact that $\set{0,1}\cap B_i=\emptyset$ for all $i$, we obtain that for any finite $X$ the following holds
\[
\set{\langle 0,0\rangle,\langle 0,1\rangle} \cap (\bigcup_{k\in\omega}\range(f_k) \cup \ulcorner X\urcorner) =\emptyset.
\]

Thus, $0\notin \set{n,m}$. By construction, we have that there is a unique pair of functions $(f_i,f_j)$ such that  $X=f_i[X_0]\cup\set{\ulcorner X_0\urcorner}$ and $Y=f_j[Y_0]\cup\set{\ulcorner Y_0\urcorner}$, with $X_0 \in \C_{n-1}$ and $Y_0 \in \C_{m-1}$. We have that $z\notin \set{\ulcorner X_0\urcorner, \ulcorner Y_0\urcorner}$. This is because

\begin{enumerate}
\item  $\ulcorner X_0\urcorner \neq \ulcorner Y_0\urcorner$,
\item and  $\set{\ulcorner X_0\urcorner, \ulcorner Y_0\urcorner} \cap (\bigcup_{k\in\omega} \range(f_k))=\emptyset$.
\end{enumerate}
\smallskip

Therefore, it must be $z\in \set{f_i[X_0]\cap f_j[Y_0]}$. If $i \neq j$ we immediately obtain a contradiction, since we know that $\range(f_i)\cap \range(f_j)=\emptyset$. Hence, $f_i=f_j$ and $f_i^{-1}(z)=f_j^{-1}(z)$  must be in $ X_0\cap Y_0$. But this would imply that $\C_{n-1}$ and $C_{m-1}$ overlap, contradicting the minimality of the pair $(n,m)$.
\end{proof}

Let $R$ and $S$ be the equivalence relations generated respectively by $\bigcup_{k\in\omega}\mathcal{C}_{2k}$ and $\bigcup_{k\in\omega}\mathcal{C}_{2k+1}$, i.e.,
\[
xR y \Leftrightarrow x=y \vee  (\exists Z \in \bigcup_{k\in\omega}\mathcal{C}_{2k})(x,y \in Z)
\]

and
\[
xS y \Leftrightarrow x=y \vee  (\exists Z \in \bigcup_{k\in\omega}\mathcal{C}_{2k+1})(x,y \in Z).
\]

Item $(2)$ of the last claim ensures that the two equivalence relations are well-defined, by guaranteeing that all equivalence classes of $R$ and $S$ are pairwise disjoint.
Moreover, as a consequence of item $(1)$ of the last claim, we obtain that all equivalence classes of $R$ and $S$ are finite, but they have arbitrary large size: any equivalence class  of  $R$ is either a singleton or has even size; all equivalence class of $S$ have odd size.
We claim that $\Spec_{\Leftrightarrow}(R,S)=\dA$.

\begin{claim}\label{lem:twocones1}
$\dA\subseteq\Spec_{\Leftrightarrow}(R,S)$.
\end{claim}

\begin{proof}
Since any bi-reducibility spectrum is upward closed, it is enough to prove that $\dB\subseteq\Spec_{\Leftrightarrow}(R,S)$. Let $\mathbf{b}_i\in \dB$. We show that $R\leq_{\mathbf{b}_i}S$ via $f_i$. If $xR y$, with $x\neq y$, then there is
$Z \in \mathcal{C}_{2k}$, for some $k$, such that $x,y\in Z$. It follows that $\set{f_i[Z]\cup \set{\ulcorner Z \urcorner}}\in \mathcal{C}_{2k+1}$. Since $\bigcup_{i\in\omega}\mathcal{C}_{2i+1}$ generates $S$, we obtain that $\set{f_i[Z]\cup \set{\ulcorner Z \urcorner}}$ forms an equivalence class of $S$ which contains $f_i[Z]$, and in particular $f_i(x)$ and $f_i(y)$. Hence, $f_i(x)S f_i(y)$ holds.

On the other hand, assume $\neg(xRy)$ and, towards a contradiction, suppose that $f_i(x) S f_i(y)$. By construction of $S$, this implies that there is $Z \in \C_{2k}$, for some $k$, such that $\set{f_i(x),f_i(y)}\subseteq \set{f_i[Z] \cup \ulcorner Z\urcorner}$. Since $\ulcorner Z \urcorner\notin \range(f_i)$, it follows that $\set{f_i(x),f_i(y)}\subseteq f_i[Z]$, and therefore $\set{x,y}\subseteq Z$. By construction of $R$ this would imply $x R y$, contradicting our assumption.

\smallskip

We proved that $\dB\subseteq \Spec_\Rightarrow(R,S)$. By a similar argument, it can be shown that, for all $i$, $S \leq_{\mathbf{b}_i} R$ via $f_i$. Thus, $\dB\subseteq \Spec_\Rightarrow(S,R)$. Since $\Spec_\Leftrightarrow(R,S)$ coincides with $\Spec_\Rightarrow(R,S)\cap \Spec_\Rightarrow(S,R)$, we conclude that $\dA\subseteq\Spec_{\Leftrightarrow}(R,S)$.
\end{proof}

It remains to show that $\Spec_{\Leftrightarrow}{(R,S)} \subseteq \dA$. We say that a total function $f$ is \emph{eventually injective} if there is $n$ such that $f$ restricted to $x>n$ is injective. Let $\mathbf{d}\in \Spec_{\Leftrightarrow}{(R,S)}$. Assume that $R\leq_{\mathbf{d}} S$ via some function $s$ and $S \leq_{\mathbf{d}} R$  via some function $t$.

\begin{claim}\label{claim:inj}
There exists an infinite set $A$, computable in $\mathbf{d}$, such that if $z \in A$ then $|[z]_R|>1$.
\end{claim}

\begin{proof}
We distinguish two cases. If $s$ (resp.\ $t$) is not eventually injective, let $A=\set{x_0, x_1,\ldots}$ be an infinite set such that, for all $k$, $s(x_{2k})=s(x_{2k+1})$ ($t(x_{2k})=t(x_{2k+1})$).
If $s$ and $t$ are both eventually injective,
 define the following sequence, for all $x$,
\[
x_0= x
\]
\[
x_{n+1}=\begin{cases}
s(x_n) &\text{ if $n$ is even,} \\

t(x_n)&\text{ if $n$ is odd,}
\end{cases}
\]

and the following function
\[
h_{x}(n)=\begin{cases}
|[x_n]_R| &\text{ if $n$ is even,} \\

|[x_n]_S| &\text{ if $n$ is odd.}
\end{cases}
\]

We claim that there exists $x$ and $m$ such that $h_x$ restricted to $y>m$ is strictly increasing. Otherwise, $s$ would map infinitely many equivalence classes of $R$ into classes of smaller size of $S$; or, vice versa, $t$ would map infinitely many equivalence classes of $S$ into classes of smaller size of $R$. In both cases, this contradicts the assumption that $s$ and $t$ are eventually injective.

Thus, let $z$ and $k$ be such that $h_z$ restricted to $y>2k$ is strictly increasing and $h_z(2k)>1$. We have that $A=\set{z_{2k}: k \in \omega}$ is a partial transversal of $R$ and each element of $A$ is in a class of size larger than $1$.
\end{proof}

From the fact that $A$ intersects no singleton of $R$, it follows that  any element of $A$ is either of the form $\langle 0, p_{B_1}(k)\rangle$, for some $k$, or the form $\langle p_{B_0}(i),p_{B_i}(y)\rangle$, for some $i$ and $y$: indeed, a number which is not in any of these forms is necessarily a singleton in $R$.
We distinguish three cases.

\begin{enumerate}
\item The set $Y=\set{p_{B_1}(k):\langle 0, p_{B_1}(k)\rangle \in X}$ is infinite: If so, we can reason in a familiar way. $Y\subseteq B_1$ is c.e.\ in $s$. By Lemma~\ref{lem:initial}, we obtain
that $s$ computes $B_1$, and therefore $\mathbf{
d} \geq \mathbf{b}_1$;
\item There is $j$ such that the set $Y_j=\set{p_{B_j}(k):\langle p_{B_0}(j), p_{B_j}(k)\rangle \in X}$ is infinite: If so, $Y_j\subseteq B_j$ is c.e.\ in $s$. By Lemma~\ref{lem:initial}, we obtain that $s$ computes $B_j$, and therefore $\mathbf{
d} \geq \mathbf{b}_j$;
\item The set $Y^*=\set{p_{B_0}(k):(\exists z)(\langle p_{B_0}(k), p_{B_j}(z)\rangle \in X)}$ is infinite: If so, $Y^*\subseteq B_0$ is c.e.\ in $s$. By Lemma~\ref{lem:initial}, we obtain that $s$ computes $B_0$, and therefore $\mathbf{d}\geq \mathbf{b}_0$.
\end{enumerate}

Therefore, $\mathbf{d}\in \set{\mathbf{c}: \mathbf{c}\geq \mathbf{b_0} \vee \mathbf{c}\geq \mathbf{b_1} \vee \mathbf{c}\geq \mathbf{b_j}}$, which means that $\mathbf{d}\in \dA$ and $\Spec_{\Leftrightarrow}(R,S)\subseteq \dA$. By recalling Claim~\ref{lem:twocones1}, we conclude that $\Spec_{\Leftrightarrow}(R,S)= \dA$.
\end{proof}

\end{document}